\documentclass[reqno,10pt]{amsart}

\usepackage{amsmath, amsthm, amssymb, enumerate}
\usepackage{color}
\usepackage{datetime}
\usepackage{fancyhdr}
\usepackage{soul}

\usepackage{amsthm, amssymb}
\usepackage{amsmath}
\usepackage[margin=1in]{geometry}

\definecolor{refkey}{gray}{.3}
\definecolor{labelkey}{rgb}{.7,0.4,0}

   \def\pdot{{\color{purple}{\hskip-.0truecm\rule[-1mm]{4mm}{4mm}\hskip.2truecm}}\hskip-.3truecm}
 \def\cor{\color{red}}
 \def\cob{\color{black}}

\begin{document}
\def\dbar{\overline\partial}
\def\comm{L}
\def\intint{\int\!\!\!\!\int}
\def\intinttext{\int\!\!\!\int}
\def\intintint{\int\!\!\!\!\int\!\!\!\!\int}
\def\intintintint{\int\!\!\!\!\int\!\!\!\!\int\!\!\!\!\int}
\def\ques{{\cor \underline{??????}\cob}}
\def\nto#1{{\coC \footnote{\em \coC #1}}}
\def\fractext#1#2{{#1}/{#2}}
\def\fracsm#1#2{{\textstyle{\frac{#1}{#2}}}}   
\def\baru{U}
\def\nnonumber{}
\def\palpha{p_{\alpha}}
\def\valpha{v_{\alpha}}
\def\Ualpha{U_{\alpha}}
\def\qalpha{q_{\alpha}}
\def\walpha{w_{\alpha}}
\def\falpha{f_{\alpha}}
\def\dalpha{d_{\alpha}}
\def\galpha{g_{\alpha}}
\def\halpha{h_{\alpha}}
\def\plusdelta{+\delta}
\def\psialpha{\psi_{\alpha}}
\def\psibeta{\psi_{\beta}}
\def\betaalpha{\beta_{\alpha}}
\def\gammaalpha{\gamma_{\alpha}}
\def\Talpha{T}
\def\TTalpha{T_{\alpha}}
\def\TTalphak{T_{\alpha,k}}
\def\falphak{f^{k}_{\alpha}}

\def\vbar{\overline v}
\def\ubar{\overline u}
\def\pbar{\overline p}
\def\qbar{\overline q}
\def\wbar{\overline w}
\def\zbar{\overline z}

\newcommand{\bv}{{u}}
\newcommand{\R}{\mathbb{R}}
\newcommand{\N}{\mathbb{N}}
\newcommand{\Z}{\mathbb{Z}}

\newcommand{\Br}{B_r(x_0)}
\newcommand{\Qr}{Q_r(x_0,t_0)}

\newcommand{\hilight}[1]{\colorbox{yellow}{#1}}

\def\pdot{}
\def\Pdot{}
\def\tdot{\fbox{\fbox{\bf\tiny I'm here; \today \ \currenttime}}}

   \baselineskip=15pt

\def\nts#1{{\color{red}\hbox{\bf ~#1~}}} 

\def\mbar{{\overline M}}
\def\tilde{\widetilde}
\newtheorem{Theorem}{Theorem}[section]
\newtheorem{Corollary}[Theorem]{Corollary}
\newtheorem{Proposition}[Theorem]{Proposition}
\newtheorem{Lemma}[Theorem]{Lemma}
\newtheorem{Remark}[Theorem]{Remark}
\newtheorem{definition}{Definition}[section]
\def\theequation{\thesection.\arabic{equation}}
\def\endproof{\hfill$\Box$\\}
\def\square{\hfill$\Box$\\}
\def\comma{ {\rm ,\qquad{}} }            
\def\commaone{ {\rm ,\qquad{}} }         
\def\dist{\mathop{\rm dist}\nolimits}    
\def\sgn{\mathop{\rm sgn\,}\nolimits}    
\def\Tr{\mathop{\rm Tr}\nolimits}    
\def\div{\mathop{\rm div}\nolimits}    
\def\curl{\mathop{\rm curl}\nolimits}    
\def\supp{\mathop{\rm supp}\nolimits}    
\def\divtwo{\mathop{{\rm div}_2\,}\nolimits}    
\def\re{\mathop{\rm {\mathbb R}e}\nolimits}    
\def\indeq{\qquad{}\!\!\!\!}                     
\def\period{.}                           
\def\semicolon{\,;}                      
\def\TT{S}


\title{An anisotropic partial regularity criterion for the Navier-Stokes equations}
\author{Igor Kukavica}
\address{Department of Mathematics, University of Southern California, Los Angeles, CA 90089}
\email{kukavica@usc.edu}

\author{Walter Rusin}
\address{Department of Mathematics, Oklahoma State University, Stillwater, OK 74078}
\email{walter.rusin@okstate.edu}

\author{Mohammed Ziane}
\address{Department of Mathematics, University of Southern California, Los Angeles, CA 90089}
\email{ziane@usc.edu}

\begin{abstract}
In this paper, we address the partial regularity of suitable weak solutions of the incompressible Navier--Stokes equations. We prove an interior regularity criterion involving only one component of the velocity. Namely, if $(u,p)$ is a suitable weak solution and a certain scale-invariant quantity involving only $u_3$ is small on a space-time cylinder $Q_r(x_0,t_0)$, then $u$ is regular at $(x_0,t_0)$.
\end{abstract}

\maketitle

\section{Introduction}\label{sec:intro}
\setcounter{equation}{0}

The goal of this paper is to address the partial regularity of solutions of the 3D Navier--Stokes equations
 \begin{align}\label{NSE}
	\partial_t u - \Delta u + \sum_{j=1}^3\partial_j (u_j u) + \nabla p 
           & = 0 \indeq \nonumber \\
	\div u & = 0 \indeq 
	  \end{align}
where $u(x,t)=(u_1(x,t),u_2(x,t),u_3(x,t))$ and $p(x,t)$ denote the unknown velocity and the pressure. 

The theory of partial regularity for the NSE ,
whose aim is to estimate the
Hausdorff dimension of the singular set and development of interior
regularity criteria, was initiated by Scheffer in \cite{S1,S2}. 
In a
classical paper~\cite{CKN}, Caffarelli, Kohn, and Nirenberg proved
that 
for a suitable weak solution
the one-dimensional parabolic Hausdorff measure (parabolic
Hausdorff length) of the singular set equals zero. 
Recall that a point is regular if there exists a neighborhood
in which $u$ is bounded (and thus H\"older continuous);
otherwise,
the point is called singular.
Their
interior regularity criterion reads as follows: There exist two constants
$\epsilon_{\text{CKN}}\in(0,1]$ and $\alpha\in(0,1)$ such that if
  \begin{align*}
	\int_{Q_1}(|u|^3 + |p|^{3/2})\;dx dt \leq \epsilon_{\text{CKN}}
  \end{align*}
then
  \begin{align*}
	\|u(x,t)\|_{C^\alpha(Q_{1/2})} <\infty
  \end{align*}
where $Q_r = \{(x,t): |x| < r, -r^2 \leq t \leq 0 \}$. 
Alternative proofs were given by 
Lin~\cite{Li},
Ladyzhenskaya and Ser\"egin~\cite{LS}, 
an author of the present paper \cite{K1,K2}, 
Vasseur~\cite{V}, 
and Wolf~\cite{W1,W2}. 
The problem of partial regularity of the solutions of the
Navier--Stokes equations has since then been addressed in various
contexts \cite{KP, RS1, RS2, RS3, Se1, Se2} and a variety of interior regularity
criteria has been proposed. In particular Wolf
proved in \cite{W2} the following: There exists
$\epsilon_\text{W}>0$ such that if
  \begin{align*}
	\int_{Q_1}|u|^3\;dx dt 
        \le \epsilon_{\text{W}}
  \end{align*}
then the solution $u(x,t)$ is regular at the point $(0,0)$.

In a recent paper \cite{WZ}, Wang and Zhang proved an
anisotropic interior regularity criterion, which states: 
For every $M>0$ there exists
$\epsilon_{\text{WZ}}(M)>0$ such that if
  \begin{align*}
	\int_{Q_1} (|u|^3 + |p|^{3/2})\;dx dt \leq M
  \end{align*}
and 
  \begin{align*}
	\int_{Q_1}|u_h|^3 \;dx dt \leq \epsilon_{\text{WZ}}(M)
  \end{align*}
where $u_h = (u_1,u_2)$, then the solution $u(x,t)$ is regular at the
point $(0,0)$. Their result can be viewed as a local version of the
component-reduction regularity. Regularity is obtained by imposing
conditions only on some components of the velocity, rather that of three. 
For a comprehensive review of such results we refer the reader to 
\cite{M,PP}
and references therein. 

The purpose of this paper is to prove an interior regularity criterion
involving only one component of the velocity. 
Using a different argument from \cite{WZ},
we prove
the following stronger
statement: For every $M>0$ there
exists a constant $\epsilon(M)>0$ such that if 
  \begin{align}
	\int_{Q_1} (|u|^3 + |p|^{3/2})\;dx dt \leq M
   \label{EQ15}
  \end{align}
and 
  \begin{align*}
           \int_{Q_1} |u_3|^3 
            \;dx dt
           \leq \epsilon(M)
  \end{align*}
then $u(x,t)$ is regular at the point $(0,0)$. 
For the statement, cf.~Theorem~\ref{thm1} below. 
Note that every suitable weak solution
satisfies \eqref{EQ15} for $M$ sufficiently large.
The contradiction
argument used to prove Theorem~\ref{thm1} may be also used to prove a
new interior regularity criterion
based on the pressure. 
Namely, in Theorem~\ref{thm4} we prove that
if \eqref{EQ15} holds and if
  \begin{equation}
    \int_{Q_1} |p|^{3/2}
            \;dx dt
    \leq \epsilon(M)   
   \label{EQ16}
  \end{equation}
then the solution is regular at $(0,0)$.

Also, as a corollary of Theorem~\ref{thm1} we obtain a
stronger version of the Leray's regularity criterion concerning
weak solutions.
Namely, by \cite{G2,Le}, if
$T$ is an epoch of irregularity, then
for any $q>3$
there is a sufficiently small $\epsilon>0$ such that
$
 \|u(\cdot,t)\|_{L^q} 
\geq \fractext{\epsilon}{(T-t)^{(1-3/q)/2}}
$ for $t<T$ sufficiently close to $T$.
Recall that $T$ is an epoch of irregularity if $T$ is a singular time
for $u$, while the times $t<T$ sufficiently close to $T$ are regular.
In Corollary~\ref{cor12} we obtain that if $T>0$ is the first
singular time,
then
for all $q\ge3$
  \begin{equation*}
    \Vert (u_1,u_2)(\cdot,t)\Vert_{L^q}
    \geq 
    \frac{M}{(T-t)^{(1-3/q)/2}}   
  \end{equation*}
or
  \begin{equation*}
    \Vert u_3(\cdot,t)\Vert_{L^q}
    \geq 
    \frac{\epsilon(M)}{(T-t)^{(1-3/q)/2}},
  \end{equation*}
for $t<T$ sufficiently close to $T$. (A similar statement holds when
$T$ is an epoch of irregularity.)
Similarly, using Theorem~\ref{thm4}, we obtain Corollary~\ref{cor2} which states that 
if $T$ is the first singular time,
then
  \begin{equation*}
   \Vert u(\cdot,t)\Vert_{L^q}
   \geq 
   \frac{M}{(T-t)^{(1-3/q)/2}}
  \end{equation*}
or
  \begin{equation*}
    \Vert p(\cdot,t)\Vert_{L^{q/2}} 
    \geq 
    \frac{\epsilon(M)}{(T-t)^{1-3/q}}.
  \end{equation*}
for $t<T$ sufficiently close to $T$.

The paper is organized as follows. In the next section we state the
main results and introduce the notation used throughout the rest of the
paper. The proof is based on a contradiction argument and
Section~\ref{sec3} contains a regularity result for the limit
system, which turns out to be the Navier-Stokes system with $u_3 \equiv 0$. We would like to note that in order to prove Corollary~\ref{cor12} and Corollary~\ref{cor2} we require explicit estimates on the solutions of the considered limit system. Therefore, we cannot directly apply the results of Neustupa, Novotn{\'y} and Penel from \cite{NP,NNP}. Consequently, we need to modify this strategy to suit our needs.
The proofs of Theorems~\ref{thm1} and~\ref{thm5} are presented in
Section~\ref{sec4}, while
Section~\ref{sec5} contains the proofs
of Theorems~\ref{thm4} and~\ref{thm6}.

\section{The notation and the main results}\label{sec2}
\setcounter{equation}{0}

Let $D$ be an open, bounded, and connected subset of 
${\mathbb R}^{3}\times (0,\infty)$. 
We assume that $(u,p)$ is a suitable weak solution in
$D$, which means\\
(i) $u\in L_{t}^{\infty}L_{x}^{2}(D) \cap L_{t}^{2}H_{x}^{1}(D)$ and
$p\in L^{3/2}(D)$,\\
(ii) the Navier-Stokes equations
\eqref{NSE}
are satisfied in the weak sense, and\\
(iii) the local energy inequality holds in $D$, i.e.,
  \begin{align}
   & \left. \int_{\R^3} |u|^2  \phi\;dx  \right|_{T}
   +
   2\intint_{{\mathbb R}^{3}\times(-\infty,T]} |\nabla u|^2\phi\;dx dt
   \nonumber\\&\indeq\indeq
    \le
    \intint_{{\mathbb R}^{3}\times(-\infty,T]}
      \biggl(
       |u|^2(\partial_t\phi+\Delta \phi)
       + (|u|^2+2p) u\cdot \nabla\phi 
      \biggr)\;dx dt
  \end{align}
for all $\phi\in C_{0}^{\infty}(D)$ such that $\phi\ge0$ in $D$
and almost all $T\in {\mathbb R}$.

Recall the following scaling
property of the Navier-Stokes equation: If $(u(x,t),p(x,t))$
is a solution, then so is
$(\lambda u(\lambda x,\lambda^2 t), \lambda^2 p(\lambda x, \lambda^{2} t))$.

Let $(x_0,t_0) \in D$. Denote by $B_r(x_0)$ the  Euclidean ball in $\R^3$ with center at $x_0$ and radius~$r>0$; we abbreviate $B_r = B_r(0)$. By $Q_r(x_0,t_0) = B_r(x_0) \times [t_0-r^2,t_0]$ we denote the parabolic cylinder in $\R^4$ labeled by the top center point $(x_0,t_0) \in D$. The following is the main result of the paper.

\begin{Theorem}\label{thm1}
	Let $(u,p)$ be a suitable weak solution of \eqref{NSE} in a neighborhood of $\overline{Q_r(x_0,t_0)} \subset D$ which satisfies
	\begin{align}\label{bound1}
		\frac{1}{r^2}
		\int_{Q_r(x_0,t_0)} 
                   \bigl(
                      |u|^3 + |p|^{3/2}
                   \bigr)
                \;dx dt\leq M
           .
	\end{align}
Then there exists $\epsilon >0$ depending on $M$ such that if
	\begin{align}
		{\frac{1}{r^2} }
		\int_{Q_r(x_0,t_0)} |u_3|^3 \;dx dt\leq \epsilon
	\end{align}
	then $u$ is regular at $(x_0,t_0)$. 
\end{Theorem}

The above theorem follows from the following stronger statement.

\begin{Theorem}\label{thm5}
For all $M,\epsilon_0,r>0$
there exist constants $\epsilon(M,\epsilon_0)>0$ and 
$\kappa(M,\epsilon_0)\in(0,1)$ with the
following
property:
If $(u,p)$ is a suitable 
weak solution of \eqref{NSE} in 
${Q_r(x_0,t_0)} $ which satisfies 
  \begin{align}
		\frac{1}{r^2}
		\int_{Q_r(x_0,t_0)}  
	               (     |u|^3 
                             +  |p|^{3/2} )
                 	\;dx dt\leq M
   \label{EQ17}
  \end{align}
and 
  \begin{align}
		\frac{1}{r^2}
		\int_{Q_r(x_0,t_0)}  
                   |u_3|^{3}
                \; dx dt
            \le \epsilon
   \label{EQ18}
  \end{align}
then 
  \begin{align}
		\frac{1}{(\kappa r)^2}
		\int_{Q_{\kappa r}(x_0,t_0)}  
	               (     |u|^3 
                             +  |p|^{3/2} )
                 	\;dx dt\leq \epsilon_0
   \period
   \label{EQ20}
  \end{align}
\end{Theorem}

As a consequence of Theorem~\ref{thm5} we may deduce the following
improvement of a Leray's result from~\cite{Le}.

\begin{Corollary}\label{cor12}
Let $(u,p)$ be Leray's weak solution defined in a neighborhood
of $[0,T]$ with $T$ as the first singularity.
Then for every $M\ge1$ there exists
$\epsilon(M)\in(0,1]$ such that
  \begin{equation}
    \Vert (u_1,u_2)(\cdot,t)\Vert_{L^q}
    \geq 
    \frac{M}{(T-t)^{(1-3/q)/2}}   
  \end{equation}
or
  \begin{equation}\label{2.4}
    \Vert u_3(\cdot,t)\Vert_{L^q}
    \geq 
    \frac{\epsilon(M)}{(T-t)^{(1-3/q)/2}},
  \end{equation}
for all $t\in(T/2,T)$ and $q\ge3$.
\end{Corollary}

Note that for $q=3$ a stronger statement has been
established in \cite{ESS}.
Also, observe that the statement  extends to the case
when $T$ is an epoch of irregularity by translating and
rescaling the time variable.

We first prove the corollary, while the proofs of
the theorems are provided in Section~\ref{sec4}.

\begin{proof}[Proof of Corollary~\ref{cor12}]
Assume that $u$ is regular on $(0,T)$ and
  \begin{equation}
    \Vert (u_1,u_2)(\cdot,t)\Vert_{L^q}
     \leq
    \frac{M}{(T-t)^{(1-3/q)/2}}   
   \comma t\in(T/2,T)
   \label{EQ13}
  \end{equation}
and
  \begin{equation}
    \Vert u_3(\cdot,t)\Vert_{L^q}
    \leq
    \frac{\epsilon}{(T-t)^{(1-3/q)/2}}
   \comma t\in(T/2,T)
   \label{EQ14}
  \end{equation}
hold for some $M\ge1$ and
$\epsilon\in(0,M]$. We claim that $T$ is regular
if $\epsilon$  is sufficiently small.
The assumptions on the velocity and 
the Calder\'on-Zygmund theorem imply
  \begin{equation}
   \|p(\cdot,t)\|_{L^{q/2}} 
    \leq 
    \frac{C M^2}{(T-t)^{(1-3/q)}}   
   \period
   \label{EQ21}
  \end{equation}
Let $x_0\in{\mathbb R}^{3}$ be arbitrary.
Using H\"older's inequality, we 
get
  \begin{align}
   \Vert u_j\Vert_{L^3(Q_{\sqrt T/2}(x_0,T))}
   \le
   C (\sqrt T)^{5/3-5/q}
   \Vert u_j\Vert_{L^q(Q_{\sqrt T/2}(x_0,T))}
   \le
   C T^{1/3} M
   \comma j=1,2
   \label{EQ22}
  \end{align}
where we used \eqref{EQ13} in the last step.
Therefore,
  \begin{equation}
   \frac{1}{(\sqrt T/2)^{2/3}}
   \Vert u_j\Vert_{L^3(Q_{\sqrt T/2}(x_0,T))}
   \le
   C M
   \comma j=1,2
   \period
   \label{EQ23}
  \end{equation}
Similarly, \eqref{EQ14} implies
  \begin{equation}
   \frac{1}{(\sqrt T/2)^{2/3}}
   \Vert u_3\Vert_{L^3(Q_{\sqrt T/2}(x_0,T))}
   \le
   C \epsilon
   \label{EQ24}
  \end{equation}
while by \eqref{EQ21}
  \begin{equation}
   \frac{1}{(\sqrt T/2)^{2/3}}
   \Vert p\Vert_{L^{3/2}(Q_{\sqrt T/2}(x_0,T))}^{1/2}
   \le
   C M
   \period
   \label{EQ25}
  \end{equation}
By Theorem~\ref{thm5}, 
there exists $\kappa\in(0,1)$ so that
if $\epsilon>0$ is sufficiently small
  \begin{align*}
        \frac{1}{(\kappa \sqrt T/2)^2}
	\int_{Q_{\kappa \sqrt T/2}(x_0,T)}
            (|u|^3 + |p|^{3/2})\;dx dt \leq \epsilon_{\text{CKN}}
   \comma x_0\in{\mathbb R}^{3}   
   \period
  \end{align*}
Using the CKN theory, this provides a uniform bound for $u$ for $t$ in a neighborhood
of $T$. By Leray's regularity criterion, this shows that the time
$T$ is regular, as claimed.
\end{proof}

The strategy used in the proofs of Theorems~\ref{thm1} and~\ref{thm5}
enables us to prove the following two theorems. 

\begin{Theorem}\label{thm4}
For every $M>0$, there exists a constant $\epsilon(M)>0$ with the
following
property:
If $(u,p)$ is a suitable 
weak solution of \eqref{NSE} in a neighborhood of
$\overline{Q_r(x_0,t_0)} \subset D$ which satisfies 
  \begin{align}
		\frac{1}{r^2}
		\int_{Q_r(x_0,t_0)}  
                 |u|^{3}
                 	\;dx dt\leq M
   \label{EQ02}	
  \end{align}
and
  \begin{align}
		\frac{1}{r^2}
		\int_{Q_r(x_0,t_0)}  
	                    |p|^{3/2} 
                \; dx dt
            \le \epsilon,
   \label{EQ04}
  \end{align}
then $u$ is regular at $(x_0,t_0)$. 
\end{Theorem}

Although certain regularity criteria involving the pressure are known 
(cf.~\cite{BG} for instance), the condition for regularity \eqref{EQ04} appears to be new. Theorem~\ref{thm4} follows in fact from a stronger result stated in Theorem~\ref{thm6}.

\begin{Theorem}\label{thm6}
For all $M,\epsilon_0,r>0$
there exist constants $\epsilon(M,\epsilon_0)>0$ and $\kappa(M,\epsilon_0)\in(0,1)$ with the
following
property:
If $(u,p)$ is a suitable 
weak solution of \eqref{NSE} in 
${Q_r(x_0,t_0)} $ which satisfies 
  \begin{align}
		\frac{1}{r^2}
		\int_{Q_r(x_0,t_0)}  
                     |u|^{3}
                 	\;dx dt\leq M
   \label{EQ17a}
  \end{align}
and
  \begin{align}
		\frac{1}{r^2}
		\int_{Q_r(x_0,t_0)}  
	                    |p|^{3/2} 
                \; dx dt
            \le \epsilon,
   \label{EQ19}
  \end{align}
then 
  \begin{align}
		\frac{1}{(\kappa r)^2}
		\int_{Q_{\kappa r}(x_0,t_0)}  
	               (     |u|^3 
                             +  |p|^{3/2} )
                 	\;dx dt\leq \epsilon_0
   \period
   \label{EQ20a}
  \end{align}
\end{Theorem}

As a consequence of Theorems~\ref{thm4} and~\ref{thm6} we deduce 
the following.

\begin{Corollary}\label{cor2}
Let $(u,p)$ be Leray's weak solution defined in a neighborhood
of $[0,T]$ with $T$ as the first singularity. 
Then for every $M \geq 1$ there exists $\epsilon(M) \in (0,1]$ such that
  \begin{equation*}
    \|u(\cdot,t)\|_{L^q} 
    \geq 
    \frac{M}{(T-t)^{(1-3/q)/2}}   
  \end{equation*}
or
  \begin{equation*}
     \|p(\cdot,t)\|_{L^{q/2}} 
    \geq 
    \frac{\epsilon(M)}{(T-t)^{(1-3/q)/2}},
  \end{equation*}
for all $t \in (T/2,T)$ and $q \geq 3$. 

\end{Corollary}

\begin{proof}[Proof of Corollary~\ref{cor2}]
Assume that $u$ is regular on $(0,T)$ and 
  \begin{equation}
    \Vert u(\cdot,t)\Vert_{L^q}
     \leq
    \frac{M}{(T-t)^{(1-3/q)/2}}   
   \comma t\in(T/2,T)
   \label{EQ13a}
  \end{equation}
and
  \begin{equation}
    \Vert p(\cdot,t)\Vert_{L^{q/2}}
    \leq
    \frac{\epsilon}{(T-t)^{(1-3/q)/2}}
   \comma t\in(T/2,T)
   \label{EQ14a}
  \end{equation}
hold for some $M\ge1$ and
$\epsilon\in(0,M]$. We claim that $T$ is regular
if $\epsilon$  is sufficiently small.
Let $x_0 \in \mathbb{R}^3$ be arbitrary. Using H\"older's inequality, we obtain
	\begin{align}
		\|u\|_{L^3(Q_{\sqrt{T}/2}(x_0,T))} 
         	\leq 
		C(\sqrt{T})^{5/3-5/q} 
		\leq
		CT^{1/3}M
         \end{align}
where we used (\ref{EQ13a}) in the last step. Thus, we get
  \begin{equation}
   \frac{1}{(\sqrt T/2)^{2/3}}
   \Vert u\Vert_{L^3(Q_{\sqrt T/2}(x_0,T))}
   \le
   C M
   \period
   \label{EQ23a}
  \end{equation}
Similarly, by \eqref{EQ14a} we have
  \begin{equation}
   \frac{1}{(\sqrt T/2)^{2/3}}
   \Vert p\Vert_{L^{3/2}(Q_{\sqrt T/2}(x_0,T))}^{1/2}
   \le
   C \epsilon
   \period
  \end{equation}
By Theorem~\ref{thm6}, 
there exists $\kappa\in(0,1)$ so that
if $\epsilon>0$ is sufficiently small
  \begin{align*}
        \frac{1}{(\kappa \sqrt T/2)^2}
	\int_{Q_{\kappa \sqrt T/2}(x_0,T)}
            (|u|^3 + |p|^{3/2})\;dx dt \leq \epsilon_{\text{CKN}}
   \comma x_0\in{\mathbb R}^{3}   
   \period
  \end{align*}
This provides a uniform bound for $u$ for $t$ in a neighborhood
of $T$. By Leray's regularity criterion, this shows that the time
$T$ is regular, as claimed.
\end{proof}


\section{The limit system}\label{sec3}
\setcounter{equation}{0}

Let $D \subset {\mathbb R}^{3} \times (0,\infty)$ be a domain. 
Consider the system 
\begin{align}\label{eqn_limit}
	 \partial_t u_i - \Delta u_i + \sum_{j=1}^2 u_j \partial_j u_i + \partial_i p & =0 \indeq \text{ in $D$}
     \comma i=1,2 \nonumber \\
	 \partial_3 p& =0 \indeq \text{ in $D$} \nonumber \\
	 \partial_1 u_1 + \partial_2 u_2& =0 \indeq \text{ in $D$}
\end{align}
where $u(x_1,x_2,x_3,t)$ and $p(x_1,x_2,x_3,t)$ are unknown.
Note that the system \eqref{eqn_limit} stems from the Navier-Stokes equations by setting $u_3 = 0$. 

Denote by $\mathcal{S}(u)$ the set of points where the solution $u(x,t)$ of \eqref{eqn_limit} is singular.
(The definition for a regular/singular point is the same as the one for the Navier-Stokes system.)
Therefore, we may conclude that the set $\mathcal{S}(u)$ is closed in $D$ and the partial regularity results regarding the Navier-Stokes equations imply that its $1$-dimensional parabolic measure (and as a consequence its $1$-dimensional Hausdorff measure) is equal to zero. 

The following theorem, addressing regularity of the limiting system \eqref{eqn_limit}, 
is the main result of this section.

\begin{Theorem}\label{thm2}
Let $(u,p)$ be a 
weak
solution of (\ref{eqn_limit}). Then $u$ is regular. 
\end{Theorem}

We note that the results of Neustupa, Novotn{\'y} and Penel from \cite{NP, NNP} are not directly applicable in the considered setting since the weak solutions do not a priori have enough regularity to justify this approach. Moreover, in order to prove Corollary~\ref{cor12} and Corollary~\ref{cor2} we require explicit estimates on the weak solution of the system \eqref{eqn_limit} which cannot be obtained using the strategy from \cite{NP,NNP}. In particular, the presented proofs do not take advantage of epochs of irregularity. 


The first step toward the proof of Theorem~\ref{thm2}, namely establishing
the regularity of the third component of the
vorticity $\omega_3=\partial_{1}u_2-\partial_{2}u_1$ stems however from the work of Neustupa, Novotn{\'y} and Penel mentioned above. 

\begin{Lemma}\label{omega3reg}
	Let $(x_0,t_0) \in D$ and let $r>0$ be such that $Q_r(x_0,t_0) \subset D$. Then
	\begin{align}
		|\omega_3|^{q/2} \in L^\infty((t_0 - \rho^2,t_0),L^2(B_{\rho}(x_0))) \cap L^2((t_0-\rho^2,t_0), H^1(B_\rho(x_0)))
	\end{align}	
	for any $q \in [2,\infty)$ and $\rho \in (0,r/2)$.
\end{Lemma}

\begin{proof}
Applying the curl operator to the system \eqref{eqn_limit} we note that $\omega_3$ satisfies the equation
\begin{align}\label{eqn_curl}
	\partial_t \omega_3 - \Delta \omega_3 + \sum_{j=1}^2 u_j \partial_j \omega_3 =0. 
\end{align}
Without loss of generality we may assume that $(x_0,t_0) = (0,0)$. We denote $B_r=B_r(x_0) $ and $Q_r=Q_r(x_0,t_0)$.
Let $\eta$ be a smooth non-negative cut-off function, supported on $Q_{r}$, $\eta \equiv 1$ on $Q_{r/2}$ and such that $\eta$ vanishes on the lateral boundary of $Q_{r}$, that is $\eta =0$ on $B_{r} \times \{-r^2\} \cup \partial B_{r} \times (-r^2,0)$. 
Fix $q \geq 2$.
Multiplying the equation~\eqref{eqn_curl} by 
$|\omega_3|^{q-2}\omega_3\eta^2$ and integrating over $Q_{r}$, we obtain the estimate
\begin{align}
   &\sup_{-r^2/4 \leq t \leq 0} \int_{B_{r/2}} |\omega_3|^q \eta^2 \;dx + \iint_{Q_{r/2}} |\nabla (|\omega_3|^{q/2})|^2\eta^2 
   \nonumber\\&\indeq
    \leq C(q) \iint_{Q_{r}}|\omega_3|^q (\eta |\Delta \eta| + |\nabla \eta|^2 + \eta |\partial_t \eta|)\;dxdt  
	  + C(q) \iint_{Q_{r}}|u||\omega_3|^q \eta |\nabla \eta|\;dxdt,
\label{loc_q1}
\end{align}
where the second term on the right has been obtained from
\begin{align*}
	& - \sum_{j=1}^3 \iint_{Q_{r}} u_j \partial_j \omega_3 |\omega_3|^{q-2}\omega_3 \eta^2 \;dxdt   
     = -\frac{1}{q} \sum_{j=1}^3 \iint_{Q_{r}} u_j \partial_j (|\omega_3|^q)\eta^2 \;dxdt 
     \nonumber\\&\indeq
         \indeq = \frac{2}{q} \sum_{j=1}^3 \iint_{Q_{r}} u_j |\omega_3|^q \eta \partial_j \eta \;dxdt
	 \leq C \iint_{Q_{r}} |u| |\omega_3|^q \eta |\nabla \eta|\;dxdt
\end{align*}	
using integration by parts and the divergence-free condition in the
second step.
This can be formally justified using a suitable mollification and passage to the limit. The estimate (\ref{loc_q1}) yields
\begin{align}\label{loc_q2}
    &\||\omega_3|^{q/2}\|^2_{L^\infty_t L^2_x(Q_{r/2})} + \|\nabla(|\omega_3|^{q/2})\|^2_{L^2(Q_{r/2})} 
    \nonumber\\&\indeq
   \leq C(q)
\bigl(
\|u\|_{L^{10/3}(Q_{r})}\||\omega_3|^{q/2}\|_{L^{20/7}(Q_{r})}^2 + \||\omega_3|^{q/2}\|_{L^2(Q_{r})}^2
\bigr)
.
\end{align}
By the Sobolev embedding and interpolation, we obtain from (\ref{loc_q2}) 
\begin{align}
	\||\omega_3|^{q/2}\|_{L^{10/3}(Q_{r/2})} \leq C(q)\|u\|_{L^{10/3}(Q_{r})}^{1/2}\||\omega_3|^{q/2}\|_{L^{20/7}(Q_{r})} + C(q) \||\omega_3|^{q/2}\|_{L^2(Q_{r})}.
\end{align}
Since $2 < 20/7 < 10/3$ we may bootstrap the estimate. Namely, from~(\ref{loc_q2}) we obtain
\begin{align}\label{loc_q3}
	\||\omega_3|^{(7/6)q/2}\|_{L^{20/7}(Q_{r/2})}^{7/6} \leq C(q)\|u\|_{L^{10/3}(Q_{r})}^{1/2}\||\omega_3|^{q/2}\|_{L^{20/7}(Q_{r})} + C(q) \||\omega_3|^{q/2}\|_{L^2(Q_{r})}.
\end{align}
For $j=1,2,\ldots$ we define the sequences $q_j$ and $r_j$
by the recursive relationships $q_{j+1} = (7/6)q_j$
and $r_{j+1} = r_j/2$.
Then,
from \eqref{loc_q3},
\begin{align}\label{loc_q4}
	\||\omega_3|^{q_{j+1}/2}\|_{L^{20/7}(Q_{r_{j+1}})}^{7/6} \leq C(q_j)\|u\|_{L^{10/3}(Q_{r_j})}^{1/2}\||\omega_3|^{q_j/2}\|_{L^{20/7}(Q_{r_j})} + C(q_j) \||\omega_3|^{q_j/2}\|_{L^2(Q_{r_j})}.
\end{align}
Starting with $q_0=2$, we get $q_j = 2(7/6)^j$ and we conclude that for any $q \in [2,\infty)$
\begin{align}
	|\omega_3|^{q/2} \in (L^\infty_t L^2_x \cap L^2_t H^1_x)(Q_{r/2})
\end{align}
for $r$ sufficiently small. Using a covering argument, we obtain
\begin{align}\label{eqn3.54}
	|\omega_3|^{q/2} \in (L^\infty_t L^2_x \cap L^2_t H^1_x)(Q_\rho)
\end{align}
for every $\rho\in(0,r/2)$ with an explicit estimate.
\end{proof}

In order to prove Theorem~\ref{thm2}, 
we also need the following auxiliary result. 

\begin{Lemma}\label{lemma3}\cite{G1}
	Let $\Omega$ be a bounded Lipschitz domain in $\R^3$. Let further $r \in (1,\infty)$ and $m \in \{0\} \cup \N$. Then there exists a linear operator $\mathcal{R}\colon W^{m,r}_0(\Omega) \to W^{m+1,r}_0(\Omega)$ with the properties
	\begin{itemize}
		\item[1.] $\div \mathcal{R}f = f$ for all $f \in W^{m,r}_0(\Omega)$ with $\int_\Omega f \;dx =0$, and
		\item[2.] there exists $C>0$ such that $\|\nabla^{j+1}\mathcal{R}f\|_{L^r(\Omega)} \leq C \|\nabla^j f\|_{L^r(\Omega)}$ for $j=1,\ldots,m$ and for all $f \in W^{m,r}_0(\Omega)$. 
	\end{itemize} 
\end{Lemma}

\begin{proof}[Proof of Theorem~\ref{thm2}]
First, note that the system \eqref{eqn_limit} may be rewritten as 
\begin{align}\label{eqn_stokes}
	 \partial_t u_1 - \omega_3 u_2 & = -\partial_1 \left(p + \frac{1}{2}|u|^2\right) + \Delta u_1 \nonumber \\
	 \partial_t u_2 + \omega_3 u_1 & = -\partial_2 \left(p + \frac{1}{2}|u|^2\right) + \Delta u_2 \\
	\partial_3 p & = 0 \nonumber \\
	\partial_1 u_1 + \partial_2 u_2 & = 0. \nonumber 
\end{align}
We show that any point $(x_0,t_0)$ is a regular point. 
By translation, we can assume without loss of generality that $(x_0,t_0)=(0,0)$. We denote $B_r=B_r(x_0)$ and $Q_r=Q_r(x_0,t_0)$.  Let $r>0$ be as in Lemma \ref{omega3reg}. Let $\eta$ be a smooth cut-off function supported on $B_{r/2}$ and such that $\eta \equiv 1$ on $B_{r/4}$. Let $v = \eta u - V$, where $V(\cdot,t) = \mathcal{R}(\nabla \eta \cdot u(\cdot,t))$, with $\mathcal{R}$ being the operator defined as in Lemma~\ref{lemma3} with $\Omega = B_{r/2}$.
Note that we have
\begin{align}
	\int_{B_{r/2}} \nabla \eta \cdot u\;dx = \int_{B_{r/2}} \div(\eta u)\;dx = \int_{\partial B_{r/2}} \eta u \cdot n \;dS =0,
\end{align}
where $n$ is the outer normal vector to $\partial B_{r/2}$ thus we can apply Lemma \ref{lemma3}. Moreover, $\div V = \nabla \eta \cdot u$ in~$Q_{r/2}$. 
 Note also that 
\begin{align}
	 V & \in L^2((-r^2,0), W^{2,2}_0(B_{r/4})) \nonumber \\
	 \partial_t V & \in L^2(B_{r/4} \times (-r^2/16,0)).
\end{align}
Moreover, the Sobolev embedding and the control over $\partial_t V$ yield that $V$ is essentially bounded on $Q_{r/4}$. 
In turn, the above defined $v$ solves the Stokes system
\begin{align}\label{eqn_local_v}
	& \partial_t v  - \Delta v + \nabla_2(\eta p + \frac{1}{2}\eta|u|^2) = (p + \frac{1}{2}|u|^2)\nabla_2 \eta - \partial_t V + \Delta V - \omega_3 V^\perp  - \omega_3 v^\perp
  .
\end{align}
Since $u$ is a weak solution, we obtain by interpolation $u \in L^5_tL^{30/11}_x(Q_{r/4})$. Thus $p \in L^{5/2}_t L^{15/11}_x(Q_{r/4})$, whence the first term on the right of (\ref{eqn_local_v}) belongs to $L^{5/2}_tL^{15/11}_x(Q_{r/4})$. The second, third, and fourth term belong to $L^2_tL^2_x(Q_{r/4})$, where for the fourth term we used the fact that $V$ is essentially bounded and the fact that by Lemma~\ref{omega3reg} applied with $q=2$ we have $\omega_3 \in L^\infty((-r^2,0),L^2(B_{r/4})) \cap L^2((-r^2,0),H^1(B_{r/4}))$ thus by interpolation $\omega_3 \in L^{10/3}_tL^{10/3}_x(Q_{r/4})$. The last term on the right of \eqref{eqn_local_v} belongs to $L^{5/2}_t L^{15/11}_x(Q_{r/4})$ since by interpolation $u,\omega_3 \in L^5_tL^{30/11}_x(Q_{r/4})$ and $V$ is essentially bounded on~$Q_{r/4}$. Therefore, in summary, the right side of \eqref{eqn_local_v} belongs to  $L^{5/2}_t L^{15/11}_x(Q_{r/4})$. The Stokes estimate (see \cite{SW}) applied to \eqref{eqn_local_v} yields $v \in L^{5/2}_t W^{2,15/11}_x(Q_{r/4})$ and thus by Sobolev embedding we obtain $v \in L^{5/2}_tL^{15}_x(Q_{r/4})$ which is a critical Serrin's regularity class.
\color{black}
\end{proof}

In order to prove Corollary~\ref{cor12} and Corollary~\ref{cor2} we need the following estimates on solutions of \eqref{eqn_limit}.
\begin{Lemma}\label{lem35}
Let $(u,p)$ be a solution of (\ref{eqn_limit}) and 
	\begin{align*}
		\frac{1}{r^2} \int_{Q_{r}(x_0,t_0)}( |u|^3 + |p|^{3/2})\;dxdt \leq M
	\end{align*}
for some $r>0$ and $(x_0,t_0)$.
Then for any $\epsilon_0\in(0,1)$ we have
	\begin{align}
	  \frac{1}{(\kappa r)^2} 
           \int_{Q_{\kappa r}(x_0,t_0)}( |u|^3 + |p|^{3/2})\;dxdt \leq \epsilon_{0}.
        \label{EQ06}
	\end{align}
for a constant $\kappa\in(0,1]$ depending only on $M$
and $\epsilon_0$.
\end{Lemma}

Consequently, under the assumptions of the theorem,
there exist $\kappa_0\in(0,1)$
and $K>0$ depending only on $M$ such
that
  \begin{equation}
   \kappa_0 r
   \Vert u\Vert_{L^\infty(Q_{\kappa_0 r})}
   \le
   K
   \period
   \label{EQ08}
  \end{equation}
The inequality \eqref{EQ06} implies \eqref{EQ08} 
with $\kappa_0=\kappa/2$
using the
standard CKN theory (cf.~\cite{CKN,K1}).

\begin{proof} 
Without loss of generality, we may assume that $(x_0,t_0)=(0,0)$. We denote $B_r(x_0) = B_r$ and $Q_r(x_0,t_0)=Q_r$. First, we note that by Lemma~\ref{omega3reg} we obtain $\omega_3 \in L^\infty((-r^2/16,0),L^2(B_{r/4})) \cap L^2((-r^2/16,0),H^1(B_{r/4}))$.
Similarly as in the proof of Theorem~\ref{thm2} we consider the abstract Stokes system (\ref{eqn_stokes}). Let $\eta$ be a smooth cut-off function supported on $B_{r/4}$ such that $\eta\equiv 1$ on $B_{r/8}$. We define $v = \eta u - V$ on $Q_{r/8}$. In particular  we obtain
\begin{align}
	& V \in L^2((-r^2/64,0), W^{2,2}_0(B_{r/8}(x_0))) \nonumber \\
	& \partial_t V \in L^2(B_{r/8}(x_0) \times (-r^2/64,0)).
\end{align}
Appropriate estimates follow from Lemma \ref{lemma3}, the Sobolev embedding theorem and interpolation.
On the other hand, the above defined $v$ solves the  Stokes system
\begin{align}\label{eqn_local_v2}
	& \partial_t v - \Delta v +  \nabla_2\left(\eta p + \frac{1}{2}\eta|u|^2\right) = \left(p + \frac{1}{2}|u|^2\right)\nabla_2 \eta - \partial_t V - \omega_3 V^\perp  + \Delta V - \omega_3 v^\perp.
\end{align}
Proceeding as in the proof of Theorem~\ref{thm2} we obtain that $v$ is in the critical Serrin's regularity class. Therefore, in order to prove \eqref{EQ06} we repeat the Stokes estimate on a smaller cylinder $Q_{r/16}$ which yields $v$ in a subcritical Serrin's regularity class. This combined with regularity properties of $V$ and \eqref{eqn3.54} gives us \eqref{EQ06} on a sufficiently small cylinder, that is on $Q_{\kappa r}$ for $\kappa \in (0,r/16)$ sufficiently small.
\end{proof}


\section{Proofs of Theorems~\ref{thm1} and~\ref{thm5}}
\label{sec4}
\setcounter{equation}{0}

In this section we present the proofs of Theorems~\ref{thm1} and~\ref{thm5}. In our considerations we use sequences of suitable weak solutions. 
In the process, we need the following compactness result.

\begin{Lemma}\label{lemma_conv}
	Let $(u^{(n)}, p^{(n)})$ be a sequence of suitable weak solutions such that 
	\begin{align}\label{bound_l1}
		\frac{1}{r^2}
		\int_{Q_r(x_0,t_0)} 
                   \bigl( 
                    |u^{(n)}|^3 + |p^{(n)}|^{3/2} 
                   \bigr)\;dx dt\leq M
          ,
	\end{align}
and let $0 < \rho < r$. Then there exists a subsequence $(u^{(n_k)}, p^{(n_k)})$ such that $u^{(n_k)} \to u$ strongly in $L^q(Q_\rho(x_0,t_0))$ for all $1 \leq q < 10/3$ and $p^{(n_k)} \rightharpoonup p$ weakly in $L^{3/2}(Q_\rho(x_0,t_0))$. 
\end{Lemma}

\begin{proof}[Proof of Lemma~\ref{lemma_conv}] Let $\phi \in
C^\infty_0(D)$ be such that $\phi \geq 0$ in $D$, $\phi =1 $ on
$B_\rho(x_0) \times (t_0-\rho^2,t_0)$ and 
$\supp(\phi) \subset Q_r(x_0,t_0)$. 
The local energy inequality for suitable weak solution yields
  \begin{align}
	& \int_{B_\rho(x_0)}|u(\cdot,t)|^2 \;dx + 2 \int_{Q_\rho(x_0,t_0)} |\nabla u|^2 \;dx dt \nonumber \\
	& \indeq \leq \int_{\Qr} |u|^2(\partial_t \phi + \Delta \phi)\;dx dt 
         +  \int_{\Qr}(|u|^2 + 2p)u \cdot \nabla \phi\;dx dt
    \comma -\rho^2 \le t\le 0
  \end{align}
H\"older's inequality and the bound \eqref{bound_l1} imply that 
there exists a constant $E>0$ (where $E=E(M,\rho)$) such that
  \begin{align}
	& \int_{B_\rho(x_0)}|u(\cdot,t)|^2 \;dx + 2 \int_{Q_\rho(x_0,t_0)} |\nabla u|^2 \;dx dt \leq E
    \comma -\rho^2 \le t\le 0
   .
  \end{align}
Possibly passing to a subsequence, we may assume that $u^{(n)} \rightharpoonup u$ in $L^2((t_0-\rho^2,t_0),H^1(B_\rho(x_0)))$ and weak-$\ast$ in $L^\infty((t_0-\rho^2,t_0),L^2(B_\rho(x_0)))$. We may also assume that $p^{(n)} \rightharpoonup p$ in $L^{3/2}(Q_\rho(x_0,t_0))$. The equations
\begin{align}
	\partial_t u^{(n)} = \Delta u^{(n)} - u^{(n)} \cdot \nabla u^{(n)} - \nabla p^{(n)} \indeq \text{in $Q_\rho(x_0,t_0)$}
\end{align}
and the weak convergence $u^{(n)} \rightharpoonup u$ in $L^2((t_0-\rho^2,t_0),H^1(B_\rho(x_0)))$ and weak-$\ast$ in $L^\infty((t_0-\rho^2,t_0),L^2(B_\rho(x_0)))$ along with the $L^{3/2}$ bound on $p^{(n)}$ imply that $\partial_t u^{(n)} \in L^{3/2}((t_0-\rho^2,t_0),(H^2_0)^\ast(B_\rho(x_0)))$ with a uniform bound
\begin{align}
	\|\partial_t u^{(n)}\|_{L^{3/2}((t_0-r^2,t_0),(H^2_0)^\ast(B_\rho(x_0)))} \leq C,
\end{align}
where the constant $C$ may depend on $E$. 
Therefore, by the Aubin-Lions compactness lemma we conclude that $u^{(n)} \to u$ strongly in $L^{3/2}(Q_\rho(x_0,t_0))$. Since $u^{(n)}$ is bounded uniformly in $L^{10/3}(Q_\rho(x_0,t_0))$, by interpolation we get that $u^{(n)} \to u$ strongly in $L^q(Q_\rho(x_0,t_0)))$ for all $1 \leq q < 10/3$.	
\end{proof}

We first prove the stronger result, namely Theorem~\ref{thm5}.

\begin{proof}[Proof of Theorem~\ref{thm5}]

Without loss of generality, we may assume that
$(x_0,t_0)=(0,0)$. Denote $Q_r=Q_r(x_0,t_0)$.
Fix $r>0$ and assume that there exists a sequence of suitable weak solutions $(u^{(n)},p^{(n)})$ with
  \begin{align}\label{bound21}
	\frac{1}{r^2}
	\int_{Q_r} (|u^{(n)}|^3 + |p^{(n)}|^{3/2})\;dx dt \leq M
  \end{align}
and
  \begin{align}\label{bound31}
     \frac{1}{r^{2}}\int_{Q_r}
     |u_3^{(n)}|^3 \;dx dt \to 0
     ,
  \end{align}
  but
   \begin{align}
		\frac{1}{(\kappa r)^2}
		\int_{Q_{\kappa r}}  
	               (     |u^{(n)}|^3 
                             +  |p^{(n)}|^{3/2} )
                 	\;dx dt > \epsilon_0
   \label{EQ33}
  \end{align}
for every $\kappa\in(0,1)$.
By Lemma \ref{lemma_conv}, we may 
divide $r$ by $2$ and assume that
$u^{(n)} \to u$ strongly in $L^3(Q_r)$ 
and $p^{(n)} \rightharpoonup p$ weakly in $L^{3/2}(Q_r)$. 
Note that $(u,p)$ solves the system (\ref{eqn_limit}). Theorem~\ref{thm2} implies that 
  \begin{align}
	r^{1/6}\|u\|_{L^6(Q_{r}(x_0,t_0))}\le M_0<\infty
  \end{align}
where $M_0$ depends only on $M$.
By rescaling we now assume that $r=1$.
For $\kappa_1 \in (0,1)$, 
which is to be determined below
we obtain
  \begin{align}
     \|u\|_{L^3(Q_{\kappa_1 })} 
       \leq C \kappa_1 ^{5/6}\|u\|_{L^6(Q_{\kappa_1})} 
       \leq C \kappa_1^{5/6}M_0
  \end{align}
using H\"older's inequality,
from where
  \begin{align}\label{conc1aa}
    \frac{1}{\kappa_1 ^{2/3}}\|u\|_{L^3(Q_{\kappa_1})} 
    \leq C \kappa_1^{1/6}
       \|u\|_{L^6(Q_{\kappa_1 })}   \leq C\kappa_1^{1/6}M_0.
  \end{align}
There exists $\kappa_0 >0$ such that for $\kappa_1 \in (0,\kappa_0]$ we have
\begin{equation}\label{EQ01zz}
	C \kappa_1 ^{1/6}M_0 \pdot
	\le
	\frac{1}{2}\epsilon_0^{1/3} \kappa_1^{1/12}.
\end{equation}
In particular, the inequalities \eqref{conc1aa} and \eqref{EQ01zz} then imply
  \begin{align}\label{cont_limitazz}
    \frac{1}{\kappa_1^{2}} \int |u|^3\;dxdt \leq \frac{1}{4}\epsilon_0 \kappa_1^{1/4}
   .
  \end{align}
Since $u^{(n)} \to u$ strongly in 
$L^3_{{\rm loc}}(Q_{1})$, 
we may choose $n$ large enough (depending on $\kappa_1$) so that 
  \begin{align}\label{EQ31}
     \frac{1}{\kappa_1^{2}} \int_{Q_{\kappa_1 }}|u^{(n)}|^3\;dx dt
     \le
	\epsilon_0 \kappa_1^{1/4}
  \end{align}
for $n$ sufficiently large. 

We rewrite the pressure equation as in~\cite{K1} as
  \begin{align}
   \Delta(\eta p^{(n)})
   &
   =
   -\partial_{ij}(\eta u_i^{(n)} u_j^{(n)})
   -u_i^{(n)} u_j^{(n)}\partial_{ij}\eta
   +\partial_{j}(u_i^{(n)} u_j^{(n)}\partial_{i}\eta)
   \nonumber\\&\indeq
   +\partial_{i}(u_i^{(n)} u_j^{(n)}\partial_{j}\eta)
   -p^{(n)}\Delta\eta
   +2\partial_{j}((\partial_{j}\eta)p^{(n)})
   \label{EQ34}
  \end{align}
where $\eta$ is a smooth cut-off function supported in
$B_{\kappa_1}$ identically $1$ on $B_{\kappa'\kappa_1}$
where $\kappa'\in (0,1/2]$
is to be determined below.
With $N=-1/4\pi|x|$ denoting the Newtonian potential,
we obtain
  \begin{align}
   \eta p^{(n)}
   &=
   R_i R_j(\eta u_i^{(n)} u_j^{(n)})
   -N*(u_i^{(n)} u_j^{(n)}\partial_{ij}\eta)
   +\partial_{j}N*(u_i^{(n)} u_j^{(n)}\partial_{i}\eta)
   \nonumber\\&\indeq
   +\partial_{i}N*(u_i^{(n)} u_j^{(n)}\partial_{j}\eta)
   -N*(p^{(n)}\Delta\eta)
   +2\partial_{j}N*((\partial_{j}\eta)p^{(n)})
   \nonumber\\&
   =
   p_1+p_2+p_3+p_4+p_5+p_6
   .
   \label{EQ09}
  \end{align}
For $p_1$, we have by the Calder\'on-Zygmund 
theorem
  \begin{equation}
   \Vert p_1\Vert_{L^{3/2}(Q_{\kappa'\kappa_1 })}   
   \le C \Vert u^{(n)}\Vert_{L^3(Q_{\kappa_1 })}^2
   .
   \label{EQ10}
  \end{equation}
For the rest of the terms, we use the fact that they
all contain derivatives
of $\eta$. This makes all the convolutions nonsingular
when $|x|\le \kappa'\kappa_1 $ (cf.~\cite{K1} or \cite{L} for details).
Using this, we obtain the estimate for
$p_2$, $p_3$, and $p_4$ which is as in \eqref{EQ10}.
For $p_5$, we have, as in \cite{K1},
  \begin{align}
 &  \Vert p_5\Vert_{L^{3/2}(Q_{\kappa'\kappa_1 })}
   \le
  {C(\kappa'\kappa_1)^2} \|p^{(n)}\|_{L^{3/2}(Q_1)}
     \le
  {C(\kappa'\kappa_1)^2} { M^{2/3}}
   \label{EQ12}
  \end{align}
  The same bound holds for $p_6$.
Summarizing, we obtain
  \begin{equation}
   \frac{1}{(\kappa'\kappa_1 )^{2/3}}
   \Vert p^{(n)}\Vert_{L^{3/2}(Q_{\kappa'\kappa_1 })}^{1/2}
   \le
   \frac{C_0}{(\kappa' \kappa_1 )^{2/3}}\Vert u^{(n)}\Vert_{L^{3}(Q_{\kappa_1 })}
   + C_0  { (\kappa' \kappa_1 )^{1/3}} M^{1/3}
   \label{EQ11a}
  \end{equation}
where $C_0$ is a constant. 
Using \eqref{EQ31} we bound the right side of \eqref{EQ11a} by
\begin{equation}
	  \frac{C_0 \epsilon_0}{(\kappa' )^{2/3}}
	  \kappa_1^{1/12}
   + C_0  { (\kappa' \kappa_1 )^{1/3}} M^{1/3} 
   = \frac{C_0 \epsilon_0}{(\kappa')^{2/3}}\kappa_1^{1/12} 
   + C_0 (\kappa'\kappa_1)^{1/3}M^{1/3}.
   \label{EQ_last1}
\end{equation}
We can choose $\kappa'$ and $\kappa_1$ 
small enough so that the right side of \eqref{EQ_last1} 
is smaller than $\epsilon_0^{1/3}/2$. 
By possibly making $\kappa_1$ smaller, 
we also have from \eqref{EQ31} that 
  \begin{align}\label{EQ31z}
      \frac{1}{(\kappa'\kappa_1 )^{2}} \int_{Q_{\kappa'\kappa_1 }}|u^{(n)}|^3\;dx dt
     \le
	\epsilon_0 (\kappa'\kappa_1)^{1/4} \le \frac{1}{2} \epsilon_0.
  \end{align}
Thus, setting $\kappa = \kappa' \kappa_1$ we get 
\begin{equation}
	\frac{1}{\kappa^{2}} \int_{Q_{\kappa }} (|u^{(n)}|^3+|p^{(n)}|^{3/2})\;dx dt  \le \epsilon_0,
\end{equation}
which leads to a contradiction with \eqref{EQ33}.
\end{proof}

\begin{proof}[Proof of Theorem~\ref{thm1}]
The statement follows from Theorem~\ref{thm5}
by setting $\epsilon_0=\epsilon_{CKN}$.
\end{proof}

\section{Proofs of Theorems~\ref{thm4} and~\ref{thm6}}
\label{sec5}
\setcounter{equation}{0}

Since Theorem~\ref{thm6} is more general, we start with it first.

\begin{proof}[Proof of Theorem~\ref{thm6}]
In order to prove Theorem~\ref{thm6}, 
assume that there exists a sequence of suitable weak solutions $(u^{(n)},p^{(n)})$ satisfying
  \begin{align}
		\frac{1}{r^2}
		\int_{Q_r(x_0,t_0)}  
	               (     |u^{(n)}|^3 
                             +  |p^{(n)}|^{3/2} )
                 	\;dx dt\leq M
   \label{EQ17b}
  \end{align}
and
  \begin{align}
		\frac{1}{r^2}
		\int_{Q_r(x_0,t_0)}  
	                    |p^{(n)}|^{3/2} 
                \; dx dt
            \to 0,
   \label{EQ19b}
  \end{align}
but
  \begin{align}
		\frac{1}{(\kappa r)^2}
		\int_{Q_{\kappa r}(x_0,t_0)}  
	               (     |u^{(n)}|^3 
                             +  |p^{(n)}|^{3/2} )
                 	\;dx dt > \epsilon_0
   \period
   \label{EQ20b}
  \end{align}
for a certain $\kappa \in (0,1)$
to be determined explicitly below. 
By Lemma~\ref{lemma_conv}, we may divide $r$ by $2$ and assume that $u^{(n)} \to u$ strongly in $L^3(Q_r(x_0,t_0))$. Note that $u$ solves the Burgers system
 \begin{align}
	\partial_t u_i - \Delta u_i 
            + \sum_{j=1}^3 u_i\partial_j u_j &= 0 \indeq \text{ in $D$}
     \comma i=1,2,3 
   \label{EQ05}
   \end{align}
with
  \begin{align}
	 \partial_1 u_1 + \partial_2 u_2 + \partial_3 u_3& =0 \indeq \text{ in $D$}.
   \label{EQ07}
  \end{align}
It is well-known that solutions of \eqref{EQ05} are regular (see e.g.~\cite{C}); alternatively, we may use 
local $L^{6}$ estimates combined with the divergence-free condition \eqref{EQ07}. Therefore, we get 
 \begin{align}
	r^{1/6}\|u\|_{L^6(Q_{r}(x_0,t_0))}\le M_0<\infty
  \end{align}
where $M_0$ depends only on $M$. By rescaling we now assume that $r=1$.

We then proceed as in the proof of Theorem~\ref{thm1}.
Namely, for $\kappa \in (0,1)$, H\"older's inequality yields
  \begin{align}
     \|u\|_{L^3(Q_\kappa(x_0,t_0))} 
       \leq C \kappa^{5/6}\|u\|_{L^6(Q_\kappa(x_0,t_0))} 
       \leq C\kappa^{5/6}M_0
  \end{align}
which implies
  \begin{align}\label{conc1}
   \frac{1}{\kappa^{2/3}}\|u\|_{L^3(Q_\kappa(x_0,t_0))} 
   \leq C \kappa^{5/6}
       \|u\|_{L^6(Q_\kappa(x_0,t_0))}   \leq C\kappa^{5/6}M_0
   . 
  \end{align}
Let $\kappa$ be sufficiently small so that 
  \begin{equation}
   C\kappa^{5/6}M_0 
   \leq 
\frac16 \epsilon_0^{1/3}
   .
   \label{EQ011aa}
  \end{equation}
The inequalities \eqref{conc1} and \eqref{EQ011aa} then imply
  \begin{align}\label{cont_limita}
    \frac{1}{\kappa^{2}} \int_{Q_\kappa(x_0,t_0)}|u|^3\;dxdt
    \leq 
   \frac16 \epsilon_{0}
   .
  \end{align}
Since $u^{(n)} \to u$ strongly in 
$L^3_{{\rm loc}}(Q_{1}(x_0,t_0))$, 
we may choose $n$ large enough so that 
  \begin{align}\label{contr1}
     \frac{1}{\kappa^{2}} \int_{Q_\kappa(x_0,t_0)}|u^{(n)}|^3\;dxdt
     \le
     \frac12 \epsilon_{0}
    .
  \end{align}
From (\ref{EQ19b}) and (\ref{contr1}) for sufficiently large $n$ we obtain
\begin{align}
  	\frac{1}{\kappa^2} \int_{Q_{\kappa}(x_0,t_0)} (|u^{(n)}|^3 + |p^{(n)}|^{3/2})\;dx dt \leq \epsilon_{0}.
  \end{align}
 which contradicts \eqref{EQ20b}.
\end{proof}

\begin{proof}[Proof of Theorem~\ref{thm4}]
Theorem~\ref{thm4} follows from Theorem~\ref{thm6} using the CKN criterion for
regularity.
\end{proof}

\bigskip

\noindent \footnotesize{\bf Acknowledgments.} I.K.\ was
supported in part by the NSF grants DMS-1311943, W.R. was supported in part by the NSF grant DMS-1311964,
while M.Z.\ was supported in part by the NSF grant DMS-1109562.
\normalfont
\normalsize


\end{document}